\documentclass[a4paper, reqno,12pt]{amsart}

\usepackage{standalone, tikz, tikz-3dplot}
\usepackage{amsmath, amsthm, amssymb,graphics }
\theoremstyle{plain}
\newtheorem{te}{Theorem}

\newtheorem{pr}[te]{Proposition}

\newtheorem{clm}[te]{Claim}
\theoremstyle{remark}
\newtheorem*{re}{Remark}
\newtheorem*{ack*}{Acknowledgment}

\def\y{{\bf y}}\def\b{{\bf b}}\def\n{{\bf n}}\def \bt{{\bf t}}
\def \rt{{\bar{t}}}

\def\tt{{\tilde{\tau}}}
\def\t{{\tau}}
\def\g{{\gamma}}
\def \k{{\kappa}}
\def\m{{\mu}}

\def\th{{\theta}}
\def\tth{{\tilde{\theta}}}

\def \a{{\alpha}}
\def \ta{{\tilde{\alpha}}}

\def\R{{\mathbb R}}

\def\C{{\mathbb C}}
\def\S{{\mathbb S}}

\def\P{{\mathbb P}}

\def\M{{\mathbb M}}
\def\A{{\mathbf A}}
\def\L{{\mathbb L}}
\def\J {{\mathbf J}}
\def\x{{\mathbf x}}

\def\Dec{{\operatorname{Dec}}}

\def\nint{\mathop{\diagup\kern-13.0pt\int}}

\def\Ic{{\mathcal I}}
\def \Rc{{\mathcal R}} 
\def \Dc{{\mathcal D}} \def \Ic{{\mathcal I}}
\def\Cc{{\mathcal C}}\def\Nc{{\mathcal N}}
\def\Tc{{\mathcal T}}
\def\Bc{{\mathcal B}}\def\Rc{{\mathcal R}}

\def\Pc{{\mathcal P}}
\def\Sc{{\mathcal S}}
\def\Lc{{\mathcal L}}

\def\emph#1{{\it #1}}

\begin{document}

		\author{Dominique Kemp}
		\address{Department of Mathematics, Indiana University,  Bloomington IN}
		\email{dekemp@iu.edu}

		\title[Decouplings for surfaces of zero curvature]{Decouplings for surfaces of zero curvature}\maketitle

\begin{abstract}
We extend the $l^2(L^p)$ decoupling theorem of Bourgain-Demeter to the full class of developable surfaces in $\R^3$. This completes the $l^2$ decoupling theory of the zero Gaussian curvature surfaces that lack planar (or umbilic) points. Of central interest to our study is the tangent surface associated to the moment curve.
\end{abstract}
\maketitle

\section{Background and the main result}
\label{s1}

Let $f: \R^n \to\C$. For a set $\tt \subset \R^n$, we shall denote by $f_\tt$ the Fourier restriction of $f$ to $\tt$: 
$$f_\tt (x) = \int_\tt \hat{f}(\xi) e^{2\pi i x\cdot \xi} d\xi.$$ In this context, we say that $f_\tt$ is \emph{Fourier supported} in $\tt$.

In ~\cite{BD3}, the authors proved the following $l^2(L^p)$ decoupling inequality for the $(n-1)$-dimensional compact paraboloid $$\P^{n-1} = \{(\xi_1, \dots, \xi_n) : (\xi_1, \dots, \xi_{n-1}) \in [-1/2,1/2]^{n-1}, \xi_n = \xi_1^2 + \cdots + \xi_{n-1}^2\}$$ with $\delta$-neighborhood given as $$\Nc_\delta(\P^{n-1}) = \{(\xi_1, \dots, \xi_n + v): (\xi_1, \dots, \xi_n) \in \P^{n-1}, v \in [-\delta, \delta]\}.$$

\begin{te}
\label{1}

For each $\delta > 0$, partition $[-1/2,1/2]^{n-1}$ into cubes $\tau$ of side length $\sim \delta^{\frac{1}{2}}$, and let $\Pc_\delta (\P^{n-1})$ be the collection of all curved ``boxes" $(\tau \times \R) \cap \Nc_\delta(\P^{n-1})$. Given any $2 \leq p \leq \frac{2(n+1)}{n-1}$ and any fixed $\epsilon > 0$, if $f$ is Fourier supported in $\Nc_\delta(\P^{n-1})$, then

\begin{equation} \label{one} \|f\|_{L^p(\R^n)} \lesssim_\epsilon \delta^{-\epsilon} (\sum_{\tt \in \Pc_\delta(\P^{n-1})} \|f_{\tt}\|_{L^p(\R^n)}^2)^{\frac{1}{2}}. \end{equation}

The constant in ~\eqref{one} is independent of $f$ and $\delta$. Also, $\Pc_\delta$ is maximal. No non-trivial refinement of it into smaller boxes can be taken for \eqref{one}. 

\end{te}

For the above inequality, it is not essential to take the domain of the graphing function for $\P^{n-1}$ as $[-1/2, 1/2]^{n-1}$. Using the well-known method of \emph{parabolic rescaling}, we may expand the domain by any fixed inverse power of $\delta$. We shall utilize this fact in Section \ref{s6}. 

In \cite{BD3}, the authors extend Theorem ~\ref{1} to compact hypersurfaces having positive principal curvatures everywhere. They are able there also to obtain the optimal $l^2(L^p)$ decoupling inequality for the cone in all dimensions. As well, the general case of compact hypersurfaces with nonzero Gaussian curvature was concluded in \cite{BD4} by Bourgain and Demeter.

What therefore remains is the case in which the hypersurface has zero Gaussian curvature at some or all of its points. As an initial step, Bourgain, Demeter, and the current author considered the real analytic surfaces of revolution in $\R^3$ in \cite{BDK}. In the current paper, we extend attention to surfaces in $\R^3$ having zero Gaussian curvature everywhere and no planar points (points at which both principal curvatures are zero). It is known that the zero curvature surfaces in $\R^3$ without planar points are the cylinders and cones extending over planar curves and also the tangent surfaces (as shown in Section 3-5 of \cite{DoC}). The tangent surface associated to a non-planar $C^2$ curve $\phi: I \rightarrow \R^3$ is defined to be $$\{\phi(t) + s\phi'(t): t \in I, s \in \R^+\}.$$

In this paper, we shall obtain the optimal decoupling inequality for compact $C^4$ tangent surfaces, thus completing the decoupling theory for the smooth, non-planar surfaces in $\R^3$ with zero Gaussian curvature.

The initial step is to prove $l^2(L^p)$ decoupling for the tangent surface associated to the moment curve \begin{equation} \label{momc} \phi(t) = (t, t^2, t^3). \end{equation} We shall label that surface by $\M$, call it the \emph{moment surface}
\begin{equation} \label{moms} \M = \{\phi(t) + s\phi'(t): t \in [-1/2,1/2], s \in [0,2]\}, \end{equation} and use the following parametrization for it
\begin{equation} \label{momp}
\x(t,s) = (t + s, t^2 + 2ts, t^3 + 3t^2s). \end{equation}

For our decoupling below, we shall need first to decompose $\M$ into $\log(1/\delta)$-many ``annuli" as follows. Define $A = \x([-1/2,1/2] \times [0, \delta^{1/3}])$ and $ A_k = \x([-1/2,1/2] \times [2^{-k}, 2^{-k+1}))$ for each $2^{-k} \geq \delta^{1/3}$. Since we can afford $O(1)$ losses of $(\log{1/\delta})$, we may apply the triangle inequality and H\"older's inequality to obtain:

\begin{equation} \label{five} \|f\|_{L^p(\R^3)} \leq (\log{1/\delta})^{1/2}(\|f_{\Nc_\delta(A)}\|_{L^p(\R^3)}^2 + \sum_{2^{-k} \geq \delta^{1/3}} \|f_{\Nc_\delta(A_k)}\|_{L^p(\R^3)}^2)^{1/2}. \end{equation}

Inequality \eqref{five} is our starting point for deriving:

\begin{te}{(Moment surface decoupling)}\label{4} For each $\delta > 0$ and each $k$ satisfying $\delta^{1/3} \leq 2^{-k} \leq 1$, let $\tau_k$ denote subintervals of length $\sim (2^k\delta)^{1/2}$ that partition $[-1/2,1/2]$. In turn, $\tt_k$ will denote the $\delta$-neighborhoods of the images of $\x$ on the sets $\t_k \times [2^{-k}, 2^{-k+1})$. As well, $\t$ will denote the intervals of length $\sim \delta^{1/3}$ that partition $[-1/2,1/2]$, with $\tt$ defined similarly as above.

Let $2 \leq p \leq 6$. For every $\epsilon > 0$, any $f$ whose Fourier support lies in $\Nc_\delta(\M)$ satisfies 
\begin{equation} \label{four} \|f\|_{L^p(\R^3)} \lesssim_\epsilon \delta^{-\epsilon} (\sum_{\tt \in \Pc_\delta(A)} \|f_\tt\|_{L^p(\R^3)}^2+ \sum_{2^{-k} \geq \delta^{1/3}}\sum_{\tt_k \in \Pc_\delta(A_k)} \|f_{\tt_k}\|_{L^p(\R^3)}^2)^{\frac{1}{2}}. \end{equation} 

\end{te}
\pagebreak

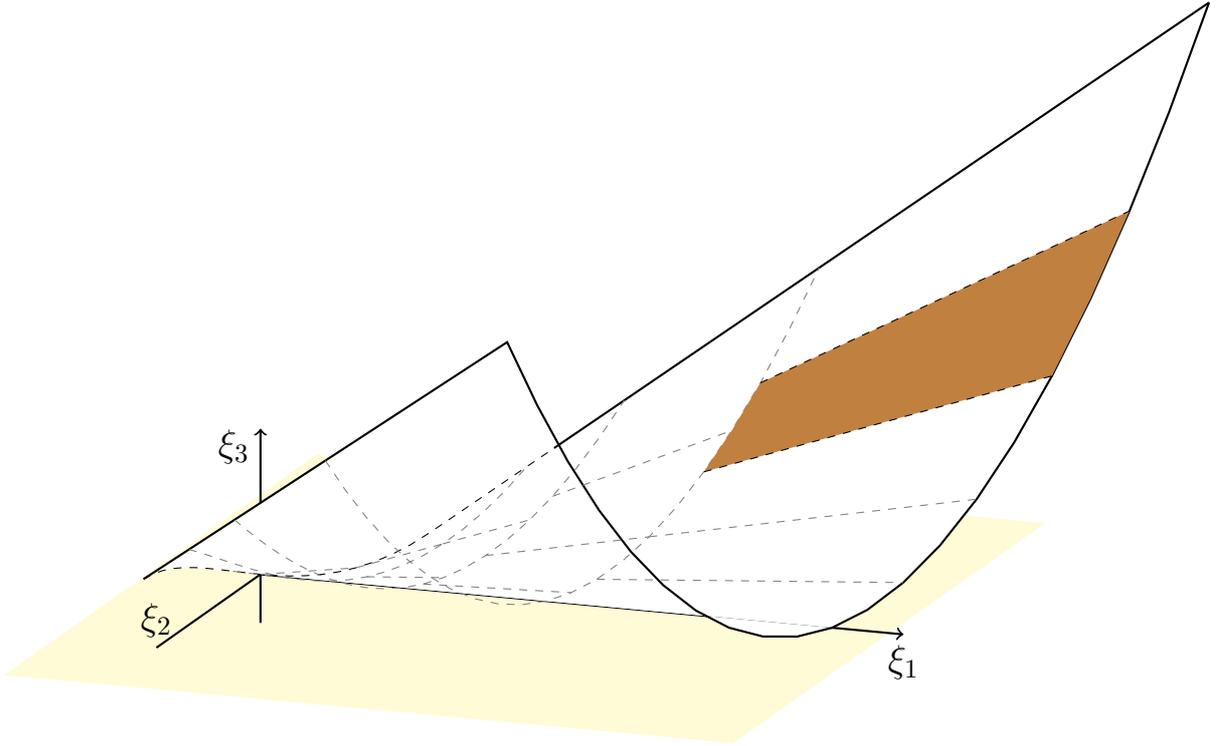
\begin{figure}
\begin{center}
\tdplotsetmaincoords{75}{20}

\begin{tikzpicture}[tdplot_main_coords, scale=4]

	\begin{scope}


     \fill[fill=yellow!20] (-0.35,-1.5, 0) --  (-0.35, 1.5, 0) --  (2.2, 1.5, 0) --  (2.2,-1.5, 0) -- cycle;


   \draw[black, thick, ->] 
(0,0,0) -- (9/4,0,0)  
    ;
   \draw[black, thick] 
(0,-1,0) -- (0,0,0)  ;

   \draw[black, thick, ->] 
(0,0,-1/6) -- (0,0,1/2)  
    ;

\fill[fill=white]
({-0.5}, {0.25}, {-0.125})
\foreach \t in {-0.5, -0.45, ..., 0}
{
--({\t}, {\t*\t}, {\t*\t*\t})
}
--({0}, {0}, {0}) -- ({2}, {0}, {0})
\foreach \t in {0, -0.05, ..., -0.5}
{
--({\t + 2}, {\t*\t + 2*\t*2}, {\t*\t*\t + 3*\t*\t*2})
}
--({-0.5+2}, {0.25 -2*0.5*2}, {-0.125 + 3*0.25*2}) -- ({-0.5}, {0.25}, {-0.125});

\fill[fill=white]
({0.5}, {0.25}, {0.125})
\foreach \t in {0.5, 0.45, ..., 0}
{
--({\t}, {\t*\t}, {\t*\t*\t})
}
--({0}, {0}, {0}) -- ({2}, {0}, {0})
\foreach \t in {0, 0.05, ..., 0.5}
{
--({\t + 2}, {\t*\t + 2*\t*2}, {\t*\t*\t + 3*\t*\t*2})
}
--({0.5+2}, {0.25 + 2*0.5*2}, {0.125+ 3*0.25*2}) -- ({0.5}, {0.25}, {0.125});

\draw[black, dashed]
({-0.5}, {0.25}, {-0.125})
\foreach \t in {-0.5, -0.45, ..., 0.5}
{
--({\t}, {\t*\t}, {\t*\t*\t})
}
--({0.5}, {0.25}, {0.125}) -- ({0.5+0.32}, {0.25 + 2*0.5*0.32}, {0.125 + 3*0.25*0.32});

\draw[black, thick]
 ({0.5+0.32}, {0.25 + 2*0.5*0.32}, {0.125 + 3*0.25*0.32}) -- ({0.5 + 2}, {0.25 + 2*0.5*2}, {0.125 + 3*0.25*2});

\draw[black, thick]
({0.5+2}, {0.25+2*0.5*2}, {0.125+3*0.25*2})
\foreach \t in {0.5, 0.45, ..., -0.5}
{
--({\t + 2}, {\t*\t + 2*\t*2}, {\t*\t*\t + 3*\t*\t*2})
}
--({-0.5 + 2}, {0.25 - 2*0.5*2}, {-0.125 + 3*0.25*2}) -- ({-0.5}, {0.25}, {-0.125});

\draw[gray, dashed]
({-0.5+1}, {0.25-2*0.5}, {-0.125 + 3*0.25})
\foreach \t in {-0.5, -0.45, ..., 0.5}
{
--({\t + 1}, {\t*\t + 2*\t}, {\t*\t*\t + 3*\t*\t})
}
--({0.5+1}, {0.25 + 2*0.5}, {0.125 + 3*0.25});

\draw[gray, dashed]
({-0.5+0.5}, {0.25-2*0.5*0.5}, {-0.125 + 3*0.25*0.5})
\foreach \t in {-0.5, -0.45, ..., 0.5}
{
--({\t + 0.5}, {\t*\t + 2*\t*0.5}, {\t*\t*\t + 3*\t*\t*0.5})
}
--({0.5+0.5}, {0.25 + 2*0.5*0.5}, {0.125 + 3*0.25*0.5});

\draw[gray, dashed]
({-0.5+0.25}, {0.25 - 2*0.5*0.25}, {-0.125 + 3*0.25*0.25})
\foreach \t in {-0.5, -0.45, ..., 0.5}
{
--({\t + 0.25}, {\t*\t + 2*\t*0.25}, {\t*\t*\t + 3*\t*\t*0.25})
}
--({0.5 + 0.25}, {0.25 + 2*0.5*0.25}, {0.125 + 3*0.25*0.25});

\draw[gray, dashed]
--({0.3}, {0.09}, {0.027})--({0.3 + 0.25}, {0.09 + 2*0.3*0.25}, {0.027 + 3*0.09*0.25});
\draw[gray, dashed]
({-0.1}, {0.01}, {-0.001}) -- ({-0.1 + 0.25}, {0.01 - 2*0.1*0.25}, {-0.001 + 3*0.01*0.25});

\draw[gray, dashed]
({0.3 + 0.25}, {0.09+2*0.3*0.25}, {0.027+3*0.09*0.25})--({0.3+0.5}, {0.09 + 2*0.3*0.5}, {0.027 + 3*0.09*0.5});

\draw[gray, dashed]
({0.1 + 0.25}, {0.01 + 2*0.1*0.25}, {0.001 + 3*0.01*0.25})--({0.1+0.5}, {0.01 + 2*0.1*0.5}, {0.001 + 3*0.01*0.5});

\draw[gray, dashed]
({0.35+0.5}, {0.35*0.35 + 2*0.35*0.5}, {0.35*0.35*0.35 + 3*0.35*0.35*0.5}) -- ({0.35+1}, {0.35*0.35 + 2*0.35*1}, {0.35*0.35*0.35 + 3*0.35*0.35*1});

\draw[gray, dashed]
({0.2+0.5}, {0.2*0.2 + 2*0.2*0.5}, {0.2*0.2*0.2 + 3*0.2*0.2*0.5}) -- ({0.2+1}, {0.2*0.2 + 2*0.2*1}, {0.2*0.2*0.2 + 3*0.2*0.2*1});

\draw[gray, dashed]
({0.05+0.5}, {0.05*0.05 + 2*0.05*0.5}, {0.05*0.05*0.05 + 3*0.05*0.05*0.5}) -- ({0.05+1}, {0.05*0.05 + 2*0.05*1}, {0.05*0.05*0.05 + 3*0.05*0.05*1});

\draw[gray, dashed]
({0.2 + 1}, {0.2*0.2 + 2*0.2}, {0.2*0.2*0.2 + 3*0.2*0.2}) -- ({0.2 + 2}, {0.2*0.2 + 2*0.2*2}, {0.2*0.2*0.2 + 3*0.2*0.2*2});

\draw[gray, dashed]
({0.1 + 1}, {0.1*0.1 + 2*0.1}, {0.1*0.1*0.1 + 3*0.1*0.1}) -- ({0.1 + 2}, {0.1*0.1 + 2*0.1*2}, {0.1*0.1*0.1 + 3*0.1*0.1*2});

\fill[fill=brown]
({0.3+1}, {0.09 + 2*0.3}, {0.027 + 3*0.09})
\foreach \t in {0.3, 0.35, 0.4}
{
--({\t + 1}, {\t*\t + 2*\t}, {\t*\t*\t + 3*\t*\t})
}
--({0.4+1}, {0.16 + 2*0.4}, {0.064 + 3*0.16}) -- ({0.4+2}, {0.16 + 2*0.4*2}, {0.064 + 3*0.16*2})
\foreach \t in {0.4, 0.35, 0.3}
{
--({\t + 2}, {\t*\t + 2*\t*2}, {\t*\t*\t + 3*\t*\t*2})
}
--({0.3+2}, {0.09 + 2*0.3*2}, {0.027 + 3*0.09*2}) -- ({0.3 + 1}, {0.09 + 2*0.3}, {0.027 + 3*0.09});

\draw[black, dashed]
({0.4 + 1}, {0.4*0.4 + 2*0.4}, {0.4*0.4*0.4 + 3*0.4*0.4}) -- ({0.4 + 2}, {0.4*0.4 + 2*0.4*2}, {0.4*0.4*0.4 + 3*0.4*0.4*2});

\draw[black, dashed]
({0.3 + 1}, {0.3*0.3 + 2*0.3}, {0.3*0.3*0.3 + 3*0.3*0.3}) -- ({0.3 + 2}, {0.3*0.3 + 2*0.3*2}, {0.3*0.3*0.3 + 3*0.3*0.3*2});

\node[above] at  ({0}, {-1}, {0}) {\large $\xi_2$};		
\node[below] at  ({9/4}, {0}, {0}) {\large $\xi_1$};	
\node[left] at  ({0}, {0}, {7/16}) {\large $\xi_3$};

		\end{scope}
			\end{tikzpicture}
\caption{$\M$ is shown along with the decoupling partition for $\delta^{1/3} = 1/4$. The projection of a $\tt$ onto $\M$ is shaded in brown.}
\end{center}
\end{figure}

\begin{re} Corresponding to the comment made prior to Theorem \ref{1}, we note here as well that the $t$-domain of $\x$ in Theorem \ref{4} may be extended by arbitrary inverse powers of $\delta$. Indeed, $\M$ with full domain for the $t$-variable projects onto the $(\xi_1, \xi_2)$-plane as a subset of the neighborhood with width $4$ of a translation of $\P^1$. Therefore, Theorem \ref{3} below enables us to decouple the $t$-domain ultimately into intervals of length $1$, at the expense of a constant having the form $C_\epsilon \delta^{-D\epsilon}$. Then, Section \ref{s4} implies that Theorem \ref{4} completes the argument. \end{re}

For the proof of Theorem ~\ref{4}, we shall need the following two theorems, which are essentially proven in ~\cite{BD3}. In particular, the proof of Theorem \ref{3} is a direct application of Fubini's Theorem and Minkowski's inequality.

\begin{te}{(Cylinder decoupling)}\label{3}

Let $\mathcal{C}yl = \P^1 \times \R$. We define the $\delta$-neighborhood of $\Cc yl$ $$\Nc_\delta(\Cc yl) = \Nc_\delta(\P^1) \times \R$$ Let $\th$ denote the elements of a partition of $[-1/2,1/2]$ into intervals of length $\sim \delta^{\frac{1}{2}}$, and let $\Pc_\delta(\Cc yl)$ be the collection of all $\tth = (\th \times \R^2) \cap \Nc_\delta(\Cc yl)$. 

Let $2 \leq p \leq 6$. For each $\epsilon > 0$ and for every $f$ that is Fourier supported in $\Nc_\delta(\Cc yl)$, 

\begin{equation} \label{three}
\|f\|_{L^p(\R^3)} \lesssim_\epsilon \delta^{-\epsilon} (\sum_{\tth \in \Pc_\delta(\Cc yl)} \|f_\tth\|_{L^p(\R^3)}^2)^{\frac{1}{2}}.
\end{equation}
\end{te}

Decoupling for the cone can be described in a variety of ways, but we will take the following perspective that involves rotating the cone $$\xi_3 = |(\xi_1,\xi_2)|$$ to the (compact) cone $\Cc$ $$\qquad \xi_3 = \frac{\xi_2^2}{2\xi_1}.$$

\begin{te} {(Cone decoupling, \cite {BD3})} \label{10} Let $\th$ be as in Theorem \ref{3}. Let $\L_\th = \{\g(t+1, t^2+2t): t \in \th; \g \in [1,2]\}$. We define $\Pc_\delta(\Cc)$ to be the collection of all sets $\tth=(\L_\th \times \R) \cap \Nc_\delta(\Cc)$. For each $2 \leq p \leq 6$, \begin{equation} \label{ten} \|f\|_{L^p} \lesssim_\epsilon \delta^{-\epsilon} (\sum_{\tth \in \Pc_\delta(\Cc)} \|f_\tth\|_{L^p}^2)^{1/2}\end{equation} for all $f$ that are Fourier supported in $\bigcup_\th \tth.$\end{te}

\begin{re} The version of Theorem \ref{10} proven in \cite{BD3} obtains a different decoupling partition $\Pc'_\delta$ for the $\delta$-neighborhood of $\Cc$. As discussed briefly in Section \ref{s5}, we may derive the partition $\Pc_\delta$ given above if every element in $\Pc_\delta$ intersects $O(1)$ many elements in $\Pc'_\delta$ and also conversely every element in $\Pc_\delta'$ intersects $O(1)$ many elements in $\Pc_\delta$.  \end{re}

\begin{figure}
\begin{center}
\tdplotsetmaincoords{80}{195}

\begin{tikzpicture}[tdplot_main_coords, scale=4]

	\begin{scope}


     \fill[fill=yellow!20] (-0.35,-1.5, 0) --  (-0.35, 1.5, 0) --  (2.2, 1.5, 0) --  (2.2,-1.5, 0) -- cycle;


   \draw[black, thick, ->] 
(-1/4,0,0) -- (1/2,0,0)  
    ;
   \draw[black, thick, ->] 
(0,-1,0) -- (0,3/4,0)  
    ;
   \draw[black, thick, ->] 
(0,0,-1/6) -- (0,0,1/2)  
    ;


     \fill[fill=white] 
     ({2},{1}, {1/2})
    \foreach \t in {0.95, 0.9,...,0}
    {
        -- ({2},{\t},{(\t*\t)/2})
    }
        -- ({2},{0},{0}) -- ({1}, {0}, {0})
 \foreach \t in {0,0.05,...,0.55}
    {
        --({1},{\t},{(\t*\t)})
    }
       -- cycle;

     \fill[fill=white] 
     ({2},{-0.2}, {(0.2)*(0.2)/2})
    \foreach \t in {-0.15, -0.1,...,0.1}
    {
        -- ({2},{\t},{(\t*\t)/2})
    }
      -- ({2},{0.1},{0}) -- ({1}, {0.05}, {0})
 \foreach \t in {0.05,0, -0.05,-0.1}
    {
        --({1},{\t},{(\t*\t)})
    }
       -- cycle;

     \fill[fill=white] 
     ({2},{-1}, {1/2})
    \foreach \t in {-0.95, -0.9,...,0}
    {
        -- ({2},{\t},{(\t*\t)/2})
    }
      -- ({2},{0},{0}) -- ({1}, {0}, {0})
 \foreach \t in {-0.05,-0.1,...,-0.5}
    {
        --({1},{\t},{(\t*\t)})
    }
      --({1},{-0.5},{(0.5*0.5)}) -- cycle;


   \draw[black, thick] 
		({1},{-0.1},{(0.1)*(0.1)})
    \foreach \t in {-0.1,-0.15,...,-0.5}
    {
        --({1},{\t},{(\t*\t)})
    };
   \draw[black, thick] 
		({1},{-0.1},{(0.1)*(0.1)})
    \foreach \t in {0,0.1,...,0.4}
    {
        --({1},{\t},{(\t*\t)})
    };

 \draw[black, thick] 
      ({1},{0.27},{(0.27)*(0.27)})
  \foreach \t in {0.3,0.35,...,0.5}
    {
        --({1},{\t},{(\t*\t)})
    }
  -- ({1},{1/2},{1/4}) -- ({2},{1},{1/2})
  \foreach \t in {1.0,0.95,...,0}
    {
        --({2},{\t},{(\t*\t)/2})
    }
     
    ;

\draw[black, dashed]
	({2}, {0}, {0})
\foreach \t in {0, -0.05, ..., -0.4}
{
--({2}, {\t}, {(\t*\t)/2})
}
;

\draw[black, thick]
	({2}, {-0.35}, {(0.35*0.35)/2})
\foreach \t in {-0.35, -0.4, ..., -1}
{
--({2}, {\t}, {(\t*\t)/2})
}
 --({2},{-1},{1/2}) -- (1, -1/2, 1/4);

\draw[black, thick]
(1,0.05,0) -- (2,0.1,0);


\fill[fill=gray]
({1.3}, {1.3*1.3-1}, {0})
\foreach \t in {1.3, 1.35, ..., 1.45}
{
--({\t}, {\t*\t -1}, {0})
}
--({1.45}, {1.45*1.45-1}, {0}) -- ({1.8}, {1.8*1.8 - 1.35*1.35}, {0})
\foreach \t in {1.8, 1.75, ..., 1.65}
{
--({\t}, {\t*\t - 1.35*1.35}, {0})
}
--({1.65}, {1.65*1.65 - 1.35*1.35}, {0}) -- ({1.3}, {1.3*1.3-1}, {0});


\fill[fill=brown]
({1.3}, {(3/5)*(1.3*1.3-1)}, {(9/25)*0.37})
\foreach \t in {1.3, 1.35, ..., 1.45}
{
--({\t}, {(3/5)*(\t*\t-1)}, {(9/25)*(\t*\t-1)*(\t*\t-1)/\t})
}
--({1.45}, {(3/5)*(1.45*1.45-1)}, {(9/25)*0.84}) --({1.8}, {(3/5)*(1.8*1.8 - 1.35*1.35)}, {(9/25)*1.12})
\foreach \t in {1.8, 1.75, ..., 1.65}
{
--({\t}, {(3/5)*(\t*\t - 1.35*1.35)}, {(9/25)*(\t*\t - 1.35*1.35)*(\t*\t - 1.35*1.35)/\t})
}
--({1.65}, {(3/5)*(1.65*1.65 - 1.35*1.35)}, {(9/25)*0.49}) -- ({1.3}, {(3/5)*(1.3*1.3-1)}, {(9/25)*0.37});


\draw[gray, thick]
	({0}, {-1}, {0})
\foreach \t in {0, 0.01, ..., 1.5}
{
--({\t}, {\t*\t - 1}, {0})
};


\draw[black, dashed]
(1, 0.4, 0.4*0.4)--(2, 0.8, 0.4*0.8);
\draw[black, dashed]
(1, 0.3, 0.3*0.3)--(2, 0.6, 0.6*0.3);
\draw[black, dashed]
(1, 0.2, 0.2*0.2)--(2, 0.4, 0.4*0.2);
\draw[black, dashed]
(1, 0.1, 0.1*0.1)--(2, 0.2, 0.2*0.1);

\node at ({1.35}, {0.66}, {0.27}) { $\tilde{\theta}$};
\node at ({1.2}, {1.1}, {0}) { $\mathbb{L}_\theta$};

\node[below] at  ({0}, {3/4}, {0}) {\large $\xi_2$};		
\node[below] at  ({1/2}, {0}, {0}) {\large $\xi_1$};	
\node[left] at  ({0}, {0}, {7/16}) {\large $\xi_3$};

		\end{scope}
			\end{tikzpicture}
\caption{Here, the projection onto $\Cc$ of a general $\tth$ with its corresponding ``shadow" $\L_\th$ is represented, as well as the caps that correspond to the original decoupling partition provided in the theorem of Bourgain and Demeter \cite{BD3}.}
\end{center}
\end{figure}
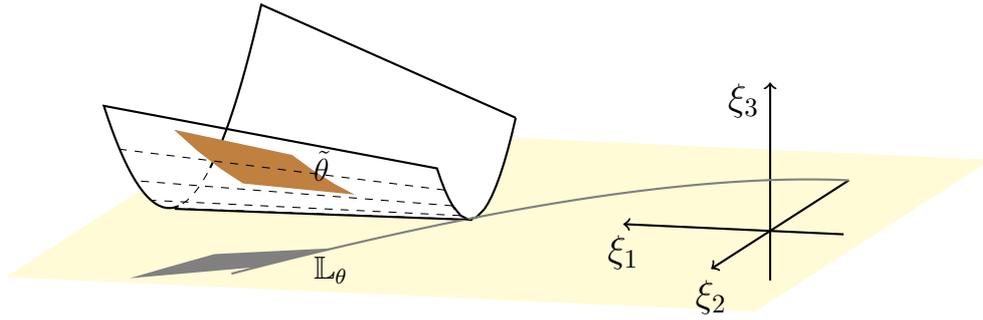

\noindent \emph{Acknowledgement.} The author would like to thank his advisor Ciprian Demeter for many encouragements and helpful discussions concerning decoupling and harmonic analysis.

\section{A translation invariance of the moment surface decoupling}
\label{s4}

One nice feature of working with the moment surface is that its decoupling is invariant with respect to translation in the $t$-variable. This is because the full moment surface is preserved by a class of corresponding affine maps of $\R^3$.

Extend the domain of $\x$ to $\R \times [0,2]$. The vertical translates of the moment surface foliate a.e. the corresponding subregion of $\R^3$ as demonstrated by the map $\Psi: \R^3 \rightarrow \R^3$:

$$\Psi(t,s,v) = (t + s, t^2 + 2ts, t^3 + 3t^2s + v)$$

The Jacobian $\mathbf{J}\Psi$ is equal to $2s$ in absolute value, and $\Psi$ is injective. Therefore, by the change of variables theorem, any integral over $\Nc_\delta = \Nc_\delta(\x(\R \times [0,2]))$ can be written as: \begin{equation} \label{cov} \int_{\Nc_\delta} g(\xi) d\xi = \int_{-\delta}^\delta \int_0^2 \int_\R g(\Psi(t,s,v))(2s) \text{ }dtds dv \end{equation}

Our translation invariance holds at any scale for $t$ and any specified range for $s$. Therefore, we shall consider subsets $\M_a$ of the full moment surface given as 
$$\M_a = \x([a-\eta, a+\eta] \times [c,d])$$ where $\eta > 0, a \in \R$, and $c < d\in [0,2]$. We shall also make use of linear maps $\A = \A_a$ described by $$\A(\xi_1, \xi_2, \xi_3) = \xi_1(1, 2a, 3a^2) + \xi_2(0,1,3a) + \xi_3(0,0,1),$$ which have determinant 1. (This fact is just a curiosity for our purposes, as it holds no bearing on what follows.)

Let us introduce the decoupling constants $\Dec_a(p,\eta)$. They are defined to be the smallest constant $C$ that satisfies

\begin{equation} \label{deca} \|f\|_{L^p(\R^3)} \leq C(\sum_{\tt \in \Pc_{\delta, a}(\M_a)} \|f_\tt\|_{L^p(\R^3)}^2)^{1/2}\end{equation}
for all $f$ that have Fourier support contained within $\Nc_{\delta, a} = \Nc_\delta(\M_a)$. $\Pc_{\delta, a} = \Pc_{\delta, a}(\M_a)$ is defined in analogy to the partition given in Theorem ~\ref{4}.

\begin{clm} $\Dec_a(p,\eta) = \Dec_b(p,\eta)$ for all $a, b \in \R$. \end{clm}

\begin{proof} It of course suffices to show that $\Dec_a(p, \eta) \leq \Dec_b(p,\eta)$. We first confirm that $\M_b$ is mapped affinely onto $\M_a$. Let $\A = \A_{a-b}$. For each $p\in \M_a$, write $$p = (t + a-b + s, (t+a-b)^2 + 2(t+a-b)s, (t+a-b)^3 + 3(t+a-b)^2s)$$ where $t \in [b-\eta, b+\eta]$. A quick check shows that $$(t+a-b+s, (t+a-b)^2 + 2(t+a-b)s, (t+a-b)^3 + 3(t+a-b)^2s) $$ $$= (t+s)(1, 2(a-b), 3(a-b)^2) + (t^2 + 2ts)(0,1,3(a-b)) + (t^3 + 3t^2s)(0,0,1)$$ $$=(t+s)\A e_1 + (t^2 + 2ts)\A e_2 + (t^3 + 3t^2s)\A e_3 + (a-b, (a-b)^2, (a-b)^3)$$

We assume of course that $f$ is Schwartz (as well as Fourier supported in $\Nc_{\delta, a}$), so that Fourier inversion applies. Define $e(r) = e^{2\pi i r}$, and let $\mathbf{a-b} = (a-b, (a-b)^2, (a-b)^3)$. By two applications of change-of-variables, we have 

\begin{eqnarray*}
|f(x)| &=& |\int_{\Nc_{\delta,a}} \hat{f}(\xi) e(x \cdot \xi) d\xi|\\&=& |\int_{-\delta}^\delta \int_c^d \int_{a-\eta}^{a+\eta} \hat{f}(\Psi(t,s,v))e(x \cdot \Psi(t,s,v))(2s) \text{ }dtdsdv| \\&=& |\int_{-\delta}^\delta \int_c^d \int_{b-\eta}^{b+\eta} \hat{f}(\Psi(t+a-b, s, v))e(x \cdot \Psi(t+a-b, s,v))(2s)\text{ }dtdsdv| \\&=&|\int_{-\delta}^\delta \int_c^d \int_{b-\eta}^{b+\eta} \hat{f}(\Psi(t+a-b, s,v))e(x \cdot \A(\Psi(t,s,v)))e(x \cdot (\mathbf{a-b}))(2s) \text{ }dtdsdv| \\&=& |\int_{-\delta}^\delta \int_c^d \int_{b-\eta}^{b+\eta} \hat{g}(\Psi(t,s,v))e(\A^Tx \cdot \Psi(t,s,v))(2s) \text{ }dtdsdv| \\&=& |\int_{\Nc_{\delta, b}} \hat{g}(\xi) e(\A^T x \cdot \xi) d\xi| 
\end{eqnarray*} where $\hat{g}(\Psi(t,s,v)) = \hat{f}(\Psi(t+a-b, s, v))$. Note that $g$ is Fourier supported in $\M_b$.

Evidently, $g(\A^Tx) = f(x)$. Therefore, \begin{eqnarray*} \|f\|_{L^p(\R^3)} &=& \|g\circ \A^T\|_{L^p(\R^3)}\\ &=& |\|g\|_{L^p(\R^3)} \\ &\leq & \Dec_b(p, \eta) (\sum_{\tt \in \Pc_{\delta, b}} \|g_\tt\|_{L^p(\R^3)}^2)^{1/2}\\ &=&\Dec_b(p, \eta) (\sum_{\tt \in \Pc_{\delta, a}} \|f_\tt\|_{L^p(\R^3)}^2)^{1/2}. \end{eqnarray*}

\end{proof}
\section{Decoupling near and far from the moment curve}
\label{s5}

In ~\eqref{five}, we achieved a partition of $\Nc_\delta(\M)$ into regions $\Nc_\delta(A)$ and $\Nc_\delta(A_k)$. Consequently, we only need to show how the partitions $\Pc_\delta(A_k)$ are obtained for functions $f$ that are Fourier supported in $\Nc_\delta(A_k)$. (We will neglect labeling $f$ by its Fourier support when it is clear from the context.) In this section, we  address only the regions $A$ and $A_0$. The other $A_k$ are addressed in Section \ref{s6}.

We first consider the region $A$. Recall that $\t$ represents the intervals of length $\delta^{1/3}$ that partition $[-1/2,1/2]$. We desire to decouple the $L^p$-norm of $f$ over sets $\tt$ that have the form $\{(t, t^2, t^3) + s(1,2t,3t^2): t \in \t, s \in [0, \delta^{1/3}]\}$. For this objective, we note that the moment curve projects down to the parabola $\P^1$. Therefore, we seek to use Theorem \ref{3}, yet we must verify two things.  Recalling the notation of Theorem \ref{3}, we claim \begin{itemize} \item[1)] $\Nc_\delta(A) \subset \Nc_{\delta^{2/3}}(\Cc yl)$. \item[2)] Each $\tt$ intersects $O(1)$ sets $\tth$, and each $\tth$ intersects $O(1)$ sets $\tt$, where $\t$ and $\th$ both have length $\delta^{1/3}$. \end{itemize}

The first claim will allow for partitioning $\Nc_{\delta}(A)$, using Theorem ~\ref{3}, into its intersections with curved boxes $\tth$ having dimensions $\sim \delta^{1/3} \times \delta^{2/3} \times \delta$. The second statement will enable a recovery of the desired boxes $\tt$ in Theorem ~\ref{4}, by way of standard Fourier projection results.

Let us prove the claims. For any $t \in [-1/2,1/2]; s \in [0, \delta^{1/3}]$, \begin{equation} \label{5.1} (t+s)^2 - (t^2 + 2ts) = s^2 \leq \delta^{2/3}.\end{equation} Therefore, $(t + s, t^2 + 2ts, t^3 + 3t^2s) \in \Nc_{\delta^{2/3}}(\Cc yl)$. By considering the $x$-coordinates of the points in $\tt$ and $\tth$, it is immediate that each $\tt$ intersects two $\tth$ and that each $\tth$ intersects at most two $\tt$.

For the remainder of this section, our task will be to obtain the partition $\Pc_{\delta}(A_0)$. This region is good for decoupling because it locally approximates a cone at sufficiently small scales. Indeed, setting $$\xi_1 = t+s, \xi_2 = t^2 + 2ts, \xi_3 = t^3 + 3t^2s, $$we have \begin{equation} \label{capp} \xi_3 = -2\xi_1^3 + 3\xi_1\xi_2 + 2(\xi_1^2 - \xi_2)^{3/2}. \end{equation}

Since $$\xi_1^2 - \xi_2 = s^2 \in [1,4]$$ and $$ \xi_1 = t+s \geq 1/2,$$ Taylor approximation applies to the fractional power in  ~\eqref{capp} yielding that $A_0$ is contained within the graph of $$ \xi_3 = \frac{3}{2}\frac{\xi_2^2}{\xi_1} + O(\frac{\xi_2^3}{\xi_1^3})$$ \begin{equation} \label{tapp} = \frac{3}{2}\frac{\xi_2^2}{\xi_1} + O(\xi_2^3) .\end{equation}

Let $\Cc'$ denote the cone described by $$\xi_3 = \frac{3}{2}\frac{\xi_2^2}{\xi_1}$$ with $\xi_1 \in [1/4,10]; \xi_2 \in [-10,10]$.

For the sequel, we note that \begin{equation} \label{sim} |\xi_2| \sim |t| \end{equation} throughout $A_0$. 

We now prove inequality ~\eqref{four} for functions $f$ that are Fourier supported in $\Nc_\delta(A_0)$. 

\begin{proof} Let $2 \leq p \leq 6$. Using the triangle inequality and H\"{o}lder, we have $$\|f\|_{L^p} \lesssim (\|f_{\Nc_\delta(\x([-1/2, 0] \times [1,2]))}\|_{L^p}^2 + \|f_{\Nc_\delta(\x([0,1/2] \times [1,2]))}\|_{L^p}^2)^{1/2}.$$ In light of Section ~\ref{s4}, it suffices to decouple $\|f_{\Nc_\delta(\x([0,1/2]\times [1,2]))}\|_{L^p}$.

For this task, we apply successive iterations of cone decoupling at increasingly smaller scales. Letting $\a_j$ denote intervals of length $\sim (1/2)^{(3/2)^j}$ that partition $[0, 1/2]$ and defining $\ta_j = \Nc_\delta(\x(\a_j \times [1,2]))$ for $0 \leq j \lesssim \lfloor \log \log 1/\delta \rfloor$, we will deduce a decoupling inequality \begin{equation} \label{stage1} \|f\|_{L^p} \lesssim (\sum_{\a_{j+1}} \|f_{\ta_{j+1}}\|_{L^p}^2)^{1/2} \end{equation} from a given inequality \begin{equation} \label{stage2} \|f\|_{L^p} \lesssim (\sum_{\a_{j}} \|f_{\ta_{j}}\|_{L^p}^2)^{1/2} \end{equation}
where $j$ satisfies $$(1/2)^{(3/2)^{j-1}} > \delta^{1/3}.$$

In deriving ~\eqref{stage1} from ~\eqref{stage2}, translation invariance implies that we need only decouple $\|f_{\ta_j}\|_{L^p}$  for $\a_j = [0, (1/2)^{(3/2)^j}]$. By ~\eqref{sim}, we know that every point of $\a_j$ has $\xi_2$-coordinate lying in $[0, C(1/2)^{(3/2)^j}]$, where $C$ is some fixed constant. Therefore, ~\eqref{tapp} implies that $\ta_j$ is contained within the $D(1/2)^{3(3/2)^j}$-neighborhood of the cone $\Cc'$ for some fixed constant $D \geq 1$. So by Theorem \ref{10}, we may partition $\a_j$ into intervals $\a_{j+1}$ having length $(1/2)^{(3/2)^{j+1}}$, achieving for each $\epsilon > 0$ \begin{equation} \label{ctomom}  \|f_{\ta_j}\|_{L^p} \leq C_{\epsilon} (2^{3(3/2)^{j}})^{\epsilon}D^{1/4}(\sum_\tth \|f_\tth\|_{L^p}^2)^{1/2}, \end{equation}
where the $\tth$ are the intersections of the $D(1/2)^{3(3/2)^j}$-neighborhoods of the sets $$\{\g(t+1, t^2+ 2t,\frac{3}{2}\frac{(t^2+2t)^2}{t+1}): \g \in \R^+, t \in \a_{j+1}\}$$
 with $\Nc_\delta(A_0)$.

It remains to recover the sets $\ta_{j+1}$ on the right side of ~\eqref{ctomom}. As at the beginning of this section, we accomplish this task by demonstrating that each $\tth$ intersects at most $O(1)$ sets $\ta_{j+1}$ and also that each $\ta_{j+1}$ intersects at most $O(1)$ sets $\tth$. It of course is sufficient to work with the $(\xi_1,\xi_2)$-projections of these sets. The slope of the ray $$\Lc_{1, t} = \{\g(t+1, t^2+2t): \g \in \R^+\}$$ is smaller than that of the line segment $$\Lc_{2,t} = \{\g(1,2t) + (t+1, t^2+2t): \g \in [0,1]\}.$$ Therefore, it is enough to show that for all $j$, \begin{equation} \label{noint} \Lc_{1, t+(1/2)^{(3/2)^{j+1}}} \cap \Lc_{2, t} = \emptyset \end{equation}  for all $t \in [0, (1/2)^{(3/2)^j}]$.

In order to verify ~\eqref{noint}, we assume that \begin{equation} \label{first} t + s = \g(t+(1/2)^{(3/2)^{j+1}}+1) \end{equation} \begin{equation} \label{sec} \g((t+(1/2)^{(3/2)^{j+1}})^2 + 2(t + (1/2)^{(3/2)^{j+1}})) \leq t^2 + 2ts \end{equation} where $t \in [0, (1/2)^{(3/2)^j}], s \in [1,2],$ and $\g \in \R^+$.

Solving for $\g$ in ~\eqref{first} and plugging the value into ~\eqref{sec}, we obtain \begin{equation} \label{thir} (t + s)((t+(1/2)^{(3/2)^{j+1}})^2 + 2(t + (1/2)^{(3/2)^{j+1}})) \leq (t^2 + 2ts)(t+(1/2)^{(3/2)^{j+1}}+1). \end{equation}

It is immediate that ~\eqref{thir} implies $$2(1/2)^{(3/2)^{j+1}} \leq (t^2s) \leq 2(1/2)^{2(3/2)^j},$$ which yields a contradiction.

Finally, the inductive proof submits $$ \|f\|_{L^p} \leq (C_\epsilon D^{1/4})^{\log_{3/2} \log 1/\delta}(2^{\sum_{j=0}^{\log_{3/2} \log 1/\delta}3(3/2)^j}\cdot 1/\delta)^{\epsilon}(\sum_{\tt \in\Pc_\delta(A_0)} \|f_\tt\|_{L^p}^2)^{1/2} $$ \begin{equation} \label{momcone}= (\log 1/\delta)^{2\log C_\epsilon D^{1/4}}(1/\delta)^{20\epsilon}(\sum_{\tt \in\Pc_\delta(A_0)} \|f_\tt\|_{L^p}^2)^{1/2}.\end{equation}

\end{proof}

\section{Decoupling for the intermediate region}
\label{s6}

Now that we have a decoupling for the region $A_0$, it is relatively straightforward to obtain the partitions $\Pc_\delta(A_k)$ for all $k > 0$. The argument will use ~\eqref{momcone} as a ``black box", combined with a simple rescaling and translation of each region $\Nc_\delta(A_k)$ into $\Nc_{2^{3k}\delta}(A_0)$.

For conciseness, we shall for the rest of this paper employ the symbol $`` \lessapprox "$ to denote inequalities of the form $``\lesssim_\epsilon \delta^{-\epsilon} "$ that hold for each $\epsilon > 0$. 

\begin{proof} 

Let $2 \leq p \leq 6$, and let $f$ be Fourier supported in $\Nc_\delta(A_k)$, where $\delta^{1/3}\leq 2^{-k} \leq 1$. According to \eqref{cov}, \begin{equation} \label{6cov} f(x) = \int_{-\delta}^ \delta \int_{2^{-k}} ^{2^{-k+1}} 2s\int_{-1/2}^{1/2} \hat{f}(\Psi(t,s,v))e(2\pi ix\cdot \Psi(t,s,v)) \text{ } dt ds  dv. \end{equation}

Rescaling both $t$ and $s$ in ~\eqref{6cov} by a factor of $2^k$, we have $$f(x) = $$ $$(2^{-k})^6\int_{-2^{3k}\delta}^{2^{3k}\delta}\int_1^2 2s\int_{-2^{k-1}}^{2^{k-1}} \hat{f}(\Psi(2^{-k}t, 2^{-k}s, 2^{-3k}v))e(2\pi i (2^{-k}x_1, 2^{-2k}x_2, 2^{-3k}x_3) \cdot \Psi(t,s,v)) \text{ } dt ds dv $$ $$ $$
so that $f$ is a composition of an invertible linear map $\mathbf{L}$, defined by $\mathbf{L}(x_1,x_2,x_3) = (2^{-k}x_1,$ $2^{-k}x_2, 2^{-3k}x_3)$, with a function $g$ whose Fourier transform is supported in $$\Nc = \Nc_{2^{3k}\delta}(\x([-2^{k-1}, 2^{k-1}] \times [1,2])).$$ The computation of \eqref{5.1} shows that \begin{equation} \label{parapprox} \Nc \subset \Nc_4(\P^1) \times \R. \end{equation}

Recall that $2^{k} \leq \delta^{-1/3}$. Therefore, as was mentioned in Section \ref{s1}, Theorem ~\ref{1} yields an efficient decoupling partition of $\Nc$ into boxes that project onto the $\xi_1$-axis as intervals of length $2$. This is essentially seen by noting that for all $\delta > 0$ and $a > 1/2$, a function $f$ that is Fourier supported in $\Nc_\delta(\{(x, x^2): x \in [-a,a]\})$ can be rewritten as $f(x_1, x_2) = a^3g(ax_1, a^2x_2)$ for some function $g$ that is Fourier supported in the $(\delta/a^2)$-neighborhood of $\P^1$ with restricted domain $[-1/2,1/2]$. Applying this fact together with Theorem ~\ref{3} and the usual combination of the triangle and H\"{o}lder's inequalities, we obtain 

\begin{equation} \label{fdec} \|f\|_{L^p} = |\text{det }\mathbf{L}|^{-1/p}\|g\|_{L^p}  \lessapprox |\text{det } \mathbf{L}|^{-1/p} (\sum_\alpha \|g_{\ta}\|_{L^p}^2)^{1/2} \end{equation} where the elements $\a$ partition $[-2^{k-1}, 2^{k-1}]$ into intervals of length $1$ and $\ta = \Nc_{2^{3k}\delta}(\x(\a \times [1,2]))$. 

Because of Section ~\ref{s4}, we may ``translate" each box $\ta$ in \eqref{fdec} to $\Nc_{2^{3k}\delta}(A_0)$. Then, the latter half of Section ~\ref{s5} applies, and we may afterward translate back to the original position of $\ta$ to get \begin{equation} \label{idec} \|g_{\ta}\|_{L^p} \lessapprox (\sum_{\t^0 \subset \a} \|g_{\tt^0}\|_{L^p}^2)^{1/2} \end{equation} 
where each $\t^0$ has length $(2^{3k}\delta)^{1/2}$, and these intervals partition $\a$.

Now each $g_{\tt^0}$ has the form $$g_{\tt^0}=$$ $$(2^{-k})^6 \int_{-2^{3k}\delta}^{2^{3k}\delta}\int_1^2 2s\int_{\tt^0} \hat{f}(\Psi(2^{-k}t, 2^{-k}s, 2^{-3k}v))e(2\pi i(x_1, x_2, x_3) \cdot\Psi(t,s,v))\text{ } dt ds dv. $$

By change of variables, \begin{eqnarray*} g_{\tt^0} &=& \int_{-\delta}^\delta \int_{2^{-k}}^{2^{-k+1}} 2s\int_\t \hat{f}(\Psi(t,s,v))e(2\pi i(x_1, x_2, x_3) \cdot \Psi(t,s,v)) \text{ } dtdsdv \\ &=& f_{\tt_k} \circ \mathbf{L}^{-1} \end{eqnarray*} where the intervals $\t_k$ have length $(2^k\delta)^{1/2}$ and partition $[-1/2,1/2]$, as desired. The proof is now complete.

\end{proof}

\section{Extension of the result}
\label{s7}

With decoupling for the moment surface now attained, the $l^2$ decoupling theory of arbitrary tangent surfaces $\S$ immediately follows. The result of this section will hold for any $\S$ that is generated by a $C^4$ regular curve having nonzero torsion throughout its domain. We remind the reader that a \emph{regular} $C^1$ curve is characterized by the non-vanishing of the curve's tangent vector throughout the domain.

Let $\phi: [-1/2,1/2] \rightarrow \R^3$ be $C^4$ and regular. We may assume (because of the curve's regularity) that $\phi$ is parametrized by arc length, i.e. $$|\phi'(t)| = 1 \qquad \forall t \in [-1/2,1/2].$$ Then, it follows that $\phi''(t)$ is orthogonal to $\phi'(t)$ for all $t$; hence, the vectors $$\bt(t) = \phi'(t)$$ $$\n(t) = \frac{\phi''(t)}{|\phi''(t)|}$$ $$\b(t) = \bt(t) \wedge \n(t)$$ form an orthonormal frame, called the \emph{Frenet trihedron}. As well, we may describe the deviations of $\phi$ from its one-dimensional and two-dimensional linear approximations using the \emph{curvature} and \emph{torsion} functions. The curvature is defined by $$\k(t) = |\phi''(t)| $$ and the torsion is given by $$\mu(t) = \b'(t) \cdot \n(t).$$ We mention for later use that nonzero torsion at a point implies nonzero curvature there as well.

We now consider the tangent surface $\S$ defined by \begin{equation} \label{6.1} \y(t,s) = \phi(t) + s\phi'(t), \qquad t \in [-1/2,1/2], s \in [0,2]. \end{equation} We shall see that locally $\phi$ looks like a rescaled rotation of the moment curve. Thus, we will obtain Theorem ~\ref{5} as a corollary of Theorem ~\ref{4}. 

Set $\bt = \bt(t_0), \n = \n(t_0),$ and $\b = \b(t_0)$. By Taylor approximation,

$$ \phi(t) = \phi(t_0) + (t-t_0-\frac{\k(t_0)^2(t-t_0)^3}{6})\bt + (\frac{\k(t_0)(t-t_0)^2}{2} + \frac{\k'(t_0)(t-t_0)^3}{6})\n -\frac{\k(t_0)\m(t_0) (t-t_0)^3}{6}\b  $$ \begin{equation} \label{frenet} + O((t-t_0)^4),\end{equation} where the constant in $O((t-t_0)^4)$ is the $C^4$ norm of $\phi$. (See Section 1.6 of \cite{DoC} for the relevant computations.) The reader may note that ~\eqref{frenet} has some resemblance to the moment curve.

In light of ~\eqref{frenet}, it is natural to describe the $\delta$-neighborhood of a general tangent surface using the binormal vectors $\b(t)$. Our point is confirmed by the fact that $$(t,s) \mapsto \b(t)$$ is a unit normal vector field on $\S\backslash \y([-1/2,1/2] \times \{0\})$ that is compatible with the parametrization $\y$. Throughout this section, $\Nc_\delta(\S)$ will denote the following set \begin{equation} \label{6.2} \{\y(t,s) + v\b(t): t \in [-1/2,1/2], s \in [0,2], v \in [-\delta, \delta]\}. \end{equation} 

\begin{te}{(Tangent surface decoupling.)} \label{5} For each $\delta^{1/3} \leq 2^{-k} \leq 1$, let $A_k$ and $\t_k$ be as in Theorem ~\ref{4} and let $A$ and $\t$ be as defined there. We define $\tt_k$ here as the $\delta$-neighborhood $($in the sense of ~\eqref{6.2}$)$ of $\y(\t_k \times [2^{-k}, 2^{-k+1}])$ and $\tt$ similarly. Also as before, $\Pc_\delta(A)$ and $\Pc_\delta(A_k)$ are comprised of the elements $\tt$ and $\tt_k$.

For each $2 \leq p \leq 6$ and for all $f$ Fourier supported in $\Nc_\delta(\S)$, \begin{equation} \label{6.3} \|f\|_{L^p(\R^3)} \lesssim_{\epsilon, \phi} \delta^{-\epsilon} (\sum_{\tt \in \Pc_\delta(A)} \|f_\tt\|_{L^p(\R^3)}^2+ \sum_{2^{-k} \geq \delta^{1/3}}\sum_{\tt_k \in \Pc_\delta(A_k)} \|f_{\tt_k}\|_{L^p(\R^3)}^2)^{\frac{1}{2}} \end{equation} with constant dependent only on $\phi$ and $\epsilon$.
\end{te}

\begin{proof} We utilize the method of Pramanik-Seeger that was employed in Section 7 of \cite{BD3}. The procedure will be relatively straightforward with the primary obstacle being posed by the ``error'' terms in ~\eqref{frenet}. 

Concerning the constant in ~\eqref{6.3}, it has dependence on the $C^4$  norm of $\phi$ and on a finite number of powers (some possibly negative) of the torsion, curvature, and the derivative of the curvature of $\phi$. Let $C_\phi$ denote 10000 times the maximum of this collection of values.

Let us define the decoupling constant for $\S$. $\Dec(\delta, p)$ will denote the smallest constant $K > 0$ such that the following inequality is true: \begin{equation} \label{6.4} \|f\|_{L^p(\R^3)} \leq K (\sum_{\tt \in \Pc_\delta(A)} \|f_\tt\|_{L^p(\R^3)}^2+ \sum_{2^{-k} \geq \delta^{1/3}}\sum_{\tt_k \in \Pc_\delta(A_k)} \|f_{\tt_k}\|_{L^p(\R^3)}^2)^{\frac{1}{2}} \end{equation} In light of the triangle inequality and H\"{o}lder, it suffices to prove that for each $\epsilon > 0$, there exists $C_\epsilon > 0$ such that \begin{equation} \label{6.5} \Dec(\delta, p) \leq C_\epsilon C_\phi^{O(1)}\delta^{-O(\epsilon)} \Dec(\delta^{3/4}, p) \end{equation} for all $0< \delta <(1/100)C_\phi^{-100}$.

Let $f$ be a function whose Fourier transform is supported within $\Nc_\delta(\S)$ ($\delta$ is as specified in \eqref{6.5}). Then (since $\delta < 1$), $\hat{f}$ is also supported in $\Nc_{\delta^{3/4}}(\S)$, so \eqref{6.4} yields \begin{equation} \label{6.6} \|f\|_{L^p(\R^3)} \leq \Dec(\delta^{3/4},p) (\sum_{\tt \in \Pc_{\delta^{3/4}}(A)} \|f_\tt\|_{L^p(\R^3)}^2+ \sum_{2^{-k} \geq \delta^{1/4}}\sum_{\tt_k \in \Pc_{\delta^{3/4}}(A_k)} \|f_{\tt_k}\|_{L^p(\R^3)}^2)^{\frac{1}{2}} \end{equation} where $A = \y([-1/2,1/2] \times [0, \delta^{1/4}))$. Let us use the notation $\tt'$ and $\tt_k'$ to denote respectively the boxes $\Nc_\delta(\y(\t \times [0, \delta^{1/3})))$ and $\Nc_\delta(\y(\t \times [2^{-k}, 2^{-k+1})))$ with $\delta^{1/3} \leq 2^{-k} < \delta^{1/4}$ (where the $\t$'s have length $\delta^{1/4}$ and partition $[-1/2,1/2]$). The triangle inequality then yields from ~\eqref{6.6} $$  \|f\|_{L^p(\R^3)} \lessapprox \Dec(\delta^{3/4},p) (\sum_{\tt'} \|f_{\tt'}\|_{L^p(\R^3)}^2 + \sum_{\delta^{1/3} \leq 2^{-k} < \delta^{1/4}} \sum_{\tt_k'} \|f_{\tt_k'}\|_{L^p(\R^3)}^2 $$ \begin{equation} \label{6.7}+ \sum_{2^{-k} \geq \delta^{1/4}}\sum_{\tt_k \in \Pc_{\delta^{3/4}}(A_k)} \|f_{\tt_k}\|_{L^p(\R^3)}^2)^{\frac{1}{2}} \end{equation}

It therefore remains to reduce the length of $\t_k$ from $(2^k\delta^{3/4})^{1/2}$ to the value $(2^k\delta)^{1/2}$, as well as that of $\t$ to the appropriate values $\delta^{1/3}$ and $(2^k\delta)^{1/2}$ respectively. At this point, ~\eqref{frenet} intervenes. Let $\t_k = [t_0, t_0 + (2^k\delta^{3/4})^{1/2}]$ and let $\k = \k(t_0), \m = \m(t_0)$. For each $t \in \t_k$ and $2^{-k} \leq s < 2^{-k+1}$, $$ \y(t,s) = \phi(t_0) + ((t-t_0)+s - \frac{\k^2(t-t_0)^3}{6} - \frac{\k^2s(t-t_0)^2}{2})\bt + (\frac{\k(t-t_0)^2}{2} + \k s(t-t_0) $$ \begin{equation} \label{6.8}+ \frac{\k'(t-t_0)^3}{6} + \frac{\k's(t-t_0)^2}{2})\n - \k\m(\frac{(t-t_0)^3}{6} + \frac{s(t-t_0)^2}{2})\b + O(\delta) \end{equation} The inequality $$2^{-k} \geq \delta^{1/4}$$ is what implies that $(t-t_0)^4, s(t-t_0)^3 \leq 2\delta$ for all $t$ and $s$ as just stated. ~\eqref{6.8} also holds for $t \in \t$ and $0 \leq s < \delta^{1/4}$.

We first secure the full decoupling for the boxes $\tt_k$. According to \eqref{6.8}, a typical point $p \in \tt_k$ has the form $$ \phi(t_0) + (\rt + s - \frac{\k^2\rt^3}{6}  - \frac{\k^2s\rt^2}{2} + O(\delta))\bt $$ \begin{equation} \label{6.12} + (\frac{\k\rt^2}{2} + \k s\rt + \frac{\k'\rt^3}{6}+ \frac{\k's\rt^2}{2} + O(\delta))\n - (\frac{\k\mu\rt^3}{6} + \frac{\k\mu s\rt^2}{2} + O(\delta))\b + O(\delta)\b(t), \end{equation}
where $\rt = t - t_0 \in [0, (2^k \delta^{3/4})^{1/2}]$. In turn, $\b(t) = \b + O(1)$, with constant dependent upon the torsion of $\phi$. Therefore, ~\eqref{6.12} simplifies to \begin{equation} \label{6.13} \phi(t_0) + (\rt + s - \frac{\k^2\rt^3}{6}  - \frac{\k^2s\rt^2}{2} + O(\delta))\bt + (\frac{\k\rt^2}{2} + \k s\rt + \frac{\k'\rt^3}{6}+ \frac{\k's\rt^2}{2} + O(\delta))\n - (\frac{\k\mu\rt^3}{6} + \frac{\k\mu s\rt^2}{2} +O(\delta))\b. \end{equation}

We endeavor next to show that $p$ lies in the $O(\delta)$-neighborhood of an affine image $\M'$ of $\M$. The idea is to view the extra terms in \eqref{6.13} as a measure of how far $\S$ deviates from $\M'$. If we can transfer the error in the $\bt$ and $\n$ components of $\y(t,s)$ to the $\b$ component, we hope that the consequent error will be $O(\delta)$.

We accomplish this scheme simply by seeking solutions $t',s'$ to the following: \begin{equation} \label{*} t'+s' = \rt + s - \frac{\k^2 \rt^3}{6} - \frac{\k^2s\rt^2}{2}+ O(\delta) \end{equation} \begin{equation} \label{**} (t')^2 +  2t's' = \rt^2 + 2s\rt + \frac{\k'\rt^3}{3\k} + \frac{\k's\rt^2}{\k} +O(\delta).\end{equation} This achievement will enable us to conclude that $p \in \Nc_{C_\phi\delta}(\A(\M))$ for the map $$\A(\xi_1, \xi_2, \xi_3) = \phi(t_0) + \xi_1 \bt + \frac{\k}{2}\xi_2 \n - \frac{\k\mu}{6} \xi_3 \b.$$

Squaring \eqref{*} and subtracting it from \eqref{**} yields $$ (s')^2 = s^2 - \frac{\k^2\rt^4}{3} - \frac{\k^2s\rt^3}{3} - \k^2s\rt^3 - \k^2s^2\rt^2 + \frac{\k^4\rt^6}{36} + \frac{\k^4s^2\rt^4}{4} + \frac{\k^4s\rt^5}{6} - \frac{\k'\rt^3}{3\k} - \frac{\k's\rt^2}{\k} +O(\delta)$$  \begin{equation} \label{6.15} = s^2 + O(s\rt^{9/5}) +O(\delta)> 0 \end{equation} where we use the power $\rt^{1/5}$ to ensure that the first constant in \eqref{6.15} is $1/4$. This is possible because $\rt \leq s$.

By a first-order Taylor approximation, ~\eqref{6.15} implies \begin{equation} \label{6.16} s' = s + s^{-1/2}O(s\rt^{9/5}) = s+O(s^{1/2}\rt^{9/5}) +O(\delta)\end{equation} \begin{equation} \label{6.17} t' = \rt  - \frac{\k^2\rt^3}{6} - \frac{\k^2s\rt^2}{2} + O(s^{1/2}\rt^{9/5}) +O(\delta) = \rt + O(s^{1/2}\rt^{9/5}) +O(\delta)\end{equation} and therefore $$ (t')^3 + 3s'(t')^2 = t^3 + 3st^2 + O(s^{3/2}\rt^{14/5}) +O(\delta)$$ \begin{equation} \label{6.18} = \rt^3 + 3s\rt^2 + O(\delta) \end{equation} as desired.

Our achievement yields that $p \in \Nc_{(\frac{\k\mu}{6})C_\phi\delta}(\A(\M))$ for the map $$\A(\xi_1, \xi_2, \xi_3) = \phi(t_0) + \xi_1 \bt + \frac{\k}{2}\xi_2 \n - \frac{\k\mu}{6} \xi_3 \b.$$Thus, we may now apply Theorem ~\ref{4} and trivial decoupling in order to obtain the caps described in the theorem, yet \emph{prima facie} this scheme only decouples $\|f_{\tt_k}\|_p$ with respect to the $s'$ and $t'$ variables. In particular, $s'$ may traverse an interval containing either $2^{-k}$ or $2^{-k+1}$ in its interior, for a given $\tt_k$. However, as justified by a simple rescaling in $s$ (similar to what was done in Section \ref{s6}), Theorem \ref{4} encompasses many analogous inequalities for decoupling partitions that are determined by the ``annuli" $\x([-1/2,1/2] \times [a2^{-k}, a2^{-k+1}])$ with $a \sim 1$. We can apply one such inequality in order to maintain that $s$ ranges from $2^{-k}$ to $2^{-k+1}$ within the smaller caps. Furthermore, the error term in \eqref{6.17} assures us that we may further partition the $t'$-intervals, if necessary, at the expense of another numerical constant. In this way, we may finally obtain the intervals $\t_k$ of the specified length $(2^k\delta)^{1/2}$. 

Now, consider $\tt_k'$. The process outlined above for $\tt_k$ works for $\tt_k'$ also, except that we have to obtain different bounds in ~\eqref{6.15}. For example, concerning $\tt_k'$, it no longer remains true that $\rt \leq s$. However, we may exploit the inequalities \begin{equation} \label{6.19} \delta^{1/3} \leq s < \delta^{1/4} \end{equation} \begin{equation} \label{6.20} \rt \leq \delta^{1/4} \end{equation} to show $\rt^3, s\rt^2 \leq \delta^{3/4}$. Using the power $\rt^{1/6}$ to cancel the powers of $\k$ and $\k'$ in ~\eqref{6.15}, we thereby obtain \begin{equation}\label{6.21} (s')^2 = s^2 + O(\delta^{1/24}\delta^{2/3}) > 0. \end{equation} 

Again by a first-order approximation, ~\eqref{6.21} gives \begin{equation} \label{6.22} s' = s + O(\delta^{1/24}\delta^{1/2}) \end{equation} and thus \begin{equation} \label{6.23} t' = \rt + O(\delta^{1/24}\delta^{1/2}). \end{equation} These identities directly lead to \begin{equation} \label{6.24} (t')^3 + 3s'(t')^2 = \rt^3 + 3s\rt^2 + O(\delta) \end{equation} where the constant in ~\eqref{6.24} is $C_\phi$. As before, Theorem ~\ref{4} gives a decoupling with respect to $t'$-intervals of length $(2^k \delta)^{1/2}$, and we may then obtain $t$-intervals of the same size in light of the error term in \eqref{6.23}.

It only remains to further decouple $\|f_{\tt'}\|_{L^p(\R^3)}$. From ~\eqref{6.13}, it is known that $\tt'$ is contained within the cylinder $\{\phi(t_0) + \xi_1\bt+ \xi_2\n + \xi_3 \b: |\xi_1| \leq 2, |\xi_2 - \frac{\k}{2}\xi_1^2| \leq 2\delta^{2/3}\}$. Therefore, Theorem ~\ref{3} accomplishes the decoupling here, and the proof is now complete.

\end{proof}

\section{Flatness}
\label{s8}

It is in this section that we address the optimality of Theorems ~\ref{4} and ~\ref{5}. We begin our commentary by mentioning the following proposition.

\begin{pr}
\label{8.6}
Let $L$ be a line segment in $\R^n$ of length $\sim 1$. For each $0 \le \delta, N^{-1}<1$, let $\Pc_{\delta,N}$ be a partition of the $\delta$-neighborhood $\Nc_\delta(L)$ of $L$ into $\sim N$ cylinders $T$  with length $N^{-1}$ and radius $\delta$.

For $p>2$, let $D(\delta,N,p)$ be the smallest constant such that
\begin{equation}
\label{8.7}
\|f\|_{L^p(\R^n)}\le D(\delta,N,p)(\sum_{T\in\Pc_{\delta,N}}\|f_T\|_{L^p(\R^n)}^2)^{1/2}
\end{equation}
holds for all $f$ Fourier supported on $\Nc_\delta(L)$. Then
$$D(\delta,N,p)\sim N^{\frac12-\frac{1}{p}},$$
and (approximate) equality in \eqref{8.7} can be achieved by using a smooth approximation of $1_{\Nc_\delta(L)}$.
\end{pr} For the proof of Proposition \ref{8.6}, the reader may find pg. 7 of \cite{BD4} to be helpful.

From this proposition, a criterion for optimality naturally follows. Namely, given an element $\tt$ in a decoupling partition $\Pc_\delta$, $\tt$ cannot be partitioned further without significant loss if it is approximately convex. Explicitly, we say that $\tt$ is flat if it contains a convex set $\mathcal{R}$ such that $\tt \subset C\mathcal{R}$ where $C\mathcal{R}$ is the enlargement of $\Rc$ about its center by a constant $C$ that is universal. 

We show that $\tt_k$ as given in Theorem ~\ref{5} satisfies the above criterion. Recall that $\tt_k = \Nc_\delta(\y(\t_k \times [2^{-k}, 2^{-k+1}]))$. We shall work with Frenet coordinates within $\tt_k$, using ~\eqref{frenet}, and obtain $\Rc$ as a polyhedron whose sides essentially follow the Frenet trihedra located at $\tt_{k_i}$.

Consider the orthogonal projections $\pi_a$ of $\R^3$ onto the tangent plane of $\S$ at $\y(a, 2^{-k})$. The Jacobian of $\pi_a \circ \y$ is found using first-order Taylor approximations to be $$\J_{\pi_a \circ \y}(t,s) = \k s + O((t-a))$$ where $\k$ is the curvature function of the curve $\phi$.  Since $\k$ is bounded away from zero and $s$ is larger than $2^{-k}$, $$\J_{\pi_{a} \circ \y}(t,s) \ne 0$$ for all $(t,s) \in \t_{k} \times [2^{-k}, 2^{-k+1}]$, so long as we take $\delta$ sufficiently small as specified in the beginning of the proof of Theorem \ref{5}. 

Let $\t_k = [a, b]$. We shall first determine the base of $\Rc$, which we call $\Rc'$. Let $\Tc$ be the tangent plane of $\S$ at $\y(a, 2^{-k})$. The coordinate system that we use for $\Tc$ places the origin at $\y(a, 2^{-k})$ and has its axes parallel to $\bt = \bt(a)$ and $\n = \n(a)$. Denote $\Sc_1 = (\pi_a \circ \y)(\t_{k} \times \{2^{-k}\}), \Sc_2 = (\pi_a \circ \y)(\t_{k} \times \{2^{-k+1}\}), \Lc_1 = (\pi_a \circ \y)(\{a\} \times [2^{-k}, 2^{-k+1}]),$ and $\Lc_2 = (\pi_a \circ \y)(\{b\} \times [2^{-k}, 2^{-k+1}])$. For $\delta > 0$ sufficiently small, I claim that the region $\Dc$ bounded by $\Sc_1, \Sc_2, \Lc_1,$ and $\Lc_2$ is contained within the image $\Ic = (\pi_a \circ \y)(\t_{k} \times [2^{-k}, 2^{-k+1}])$. Of course, $\Ic \subset \Dc$ since the interior of the former is a connected set intersecting the interior of $\Dc$.

Assume for the purpose of contradiction that $\Dc\backslash \Ic \ne \emptyset$. Then, there exists a maximal open ball $\Bc \subset \Dc \backslash \Ic$, in light of the continuity of $\pi_a \circ \y$. Necessarily, some point $p$ in the boundary of $\Bc$ lies in $\Dc \cap \Ic$. But then, using \eqref{frenet} as a reference, we have a contradiction by the Inverse Function Theorem since $\pi_a \circ \y$ has nonsingular derivative throughout $\t_{k} \times [2^{-k}, 2^{-k+1}]$. 

Two sides of $\Rc'$ are given by $\Lc_1$ and $\Lc_2$. The other two sides may be taken as any two line segments contained in $\Dc$ that orthogonally intersect $\Lc_1$ and that are also separated by $\sim 2^{-k}$. For the completion of $\Rc$, it remains merely to show that $\y(\t_{k} \times [2^{-k}, 2^{-k+1}])$ lies within $O(\delta)$ of $\Tc$. But this is seen from the fact that $\rt^3, s\rt^2 \leq 2\delta$ for all $\rt \in [0, (2^k\delta)^{1/2}], s \in [2^{-k}, 2^{-k+1}]$ with $2^{-k} \geq \delta^{1/3}$.

The demonstration for $\t$ is similar.

\end{document}